%% file: one-relator-manifolds_071815.tex
\documentclass[12pt]{amsart}
 
 \usepackage[inner=25mm, outer=25mm, textheight=225mm]{geometry}

\usepackage{amsmath,graphics,graphicx, color, amsthm}
\usepackage{amsfonts,amssymb}
\usepackage{xy}
\usepackage{subcaption}
\usepackage{thm-restate}

\theoremstyle{plain}
\newtheorem{theorem}{Theorem}[section]
\newtheorem{lemma}[theorem]{Lemma}
\newtheorem{corollary}[theorem]{Corollary}
\newtheorem{proposition}[theorem]{Proposition}
\newtheorem{conjecture}[theorem]{Conjecture}
\newtheorem{question}[theorem]{Question}

\theoremstyle{remark}
\newtheorem*{claim}{Claim}

\newcommand{\tmfrac}[2]{\mbox{\large$\frac{#1}{#2}$}} % tiny medium frac

\def\be{\begin{equation}}
\def\ee{\end{equation}}
\def\co{\colon}
\def\RR{\mathcal{R}}

\def\R{\mathbb{R}}

\def\gl{\op{GL}}

\def\Q{\mathbb{Q}}\def\K{\mathbb{K}}\def\F{\mathbb{F}}
\def\id{\op{id}}
\def\Z{\mathbb{Z}}\def\C{\mathbb{C}}

\def\N{\mathbb{N}}\def\l{\lambda}
\def\part{\partial}\def\ll{\langle}\def\rr{\rangle}

\def\a{\alpha}

\def\bp{\begin{pmatrix}}
\def\sm{\setminus}\def\ep{\end{pmatrix}}
\def\bn{\begin{enumerate}}

\def\en{\end{enumerate}}\def\ba{\begin{array}}
\def\ea{\end{array}}

\def\a{\alpha}
\def\b{\beta}\def\ti{\tilde}
\def\hull{\op{conv}}

\def\fr12{\frac{1}{2}}
  \def\im{\op{Im}}

\def\ker{\op{Ker}}

\def\hom{\op{Hom}}

\def\deg{\op{deg}}

\def\degphi{\deg_\phi}
\def\G{\Gamma}
\def\ol{\overline}

\def\op{\operatorname}

\def\K{\mathbb{K}}

\def\th{\op{th}}

\def\cmtbf#1{} \def\cmt#1{}

\def\NN{\mathcal{N}}
\def\PP{\mathcal{P}}
\def\EEAA{\mathcal{EA}}
\def\TTEEAA{\mathcal{TEA}}

\def\TT{\mathcal{T}}
\def\MM{\mathcal{M}}
\def\QQ{\mathcal{Q}}

\def\VV{\mathcal{V}}
\def\WW{\mathcal{W}}

\def\CC{\mathcal{C}}
\def\wti{\widetilde}

\def\sym{{\op{sym}}}

\def\TT{\mathcal{T}}
\def\th{\op{th}}

\begin{document}

\title{Thurston norm via Fox calculus}

\author{Stefan Friedl}
\address{Fakult\"at f\"ur Mathematik\\ Universit\"at Regensburg\\   Germany}
\email{sfriedl@gmail.com}

\author{Kevin Schreve}
\address{Department of Mathematics\\
University of Wisconsin-Milwaukee\\   Milwaukee, WI 53211-3029\\USA}
\email{kschreve@uwm.edu}

\author{Stephan Tillmann}
\address{School of Mathematics and Statistics\\ The University of Sydney\\ NSW 2006\\ Australia} 
\email{stephan.tillmann@sydney.edu.au}
%\date{\today}
\def\subjclassname{\textup{2000} Mathematics Subject Classification}
\expandafter\let\csname subjclassname@1991\endcsname=\subjclassname \expandafter\let\csname
subjclassname@2000\endcsname=\subjclassname 


\begin{abstract}
In 1976 Thurston associated to a $3$--manifold $N$ a marked polytope in $H_1(N;\R),$ which measures the minimal complexity of surfaces representing homology classes and determines all fibered classes in $H^1(N;\R)$.
Recently the first and the last author  associated to a presentation $\pi$ with two generators and one relator a marked polytope in $H_1(\pi;\R)$ and showed that it determines the Bieri--Neumann--Strebel invariant of $\pi$.
In this paper, we show that if the fundamental group of a  3--manifold $N$ admits such a presentation $\pi$, then the corresponding marked polytopes in $H_1(N;\R)=H_1(\pi;\R)$ agree.
\end{abstract}
 
\maketitle
%==================================================================
\section{Summary of results}
Throughout this paper all $3$--manifolds are compact, connected and orientable. Suppose $N$ is a 3--manifold.
In 1976 Thurston \cite{Th86} introduced a seminorm $x_N$ on $H^1(N;\mathbb{R})$, henceforth referred to as the \emph{Thurston norm}, which is a natural measure of the complexity of surfaces dual to integral classes.
A class $\phi\in H^1(N;\R)$ is \emph{fibered} if $\phi$ can be represented by a non-degenerate closed 1-form.
If $\phi$ is integral, then $\phi$ is fibered if and only if it is induced by a surface bundle $N\to S^1$.
We refer to \ref{section:thurstonnorm} for details.

Thurston \cite{Th86} showed that the information on the Thurston seminorm and the fibered classes can be encapsulated in terms of a \emph{marked polytope}.
A marked polytope is a polytope in a vector space together with a (possibly empty) set of marked vertices. In order to state Thurston's result precisely we need one more definition. Given a  polytope  in a vector space $V$ we say that a homomorphism $\phi\in \hom(V,\R)$ \emph{pairs maximally with the vertex $v$} if $\phi(v)>\phi(w)$ for all other vertices $w\ne v$ .
In this language, the main result of Thurston \cite{Th86} can be stated as follows.

\begin{theorem}[Thurston]\label{thm:thurston}
Let $N$ be a 3--manifold. There exists a unique symmetric marked polytope $\MM_N$ in $H_1(N;\R)$ with vertices in $H_1(N;\Z)/\mbox{torsion}\subset H_1(N;\R)$ such that for any $\phi\in H^1(N;\R)=\hom(\pi_1(N),\R)$ we have 
\[ x_N(\phi)=\max\{ \phi(p)-\phi(q)\,|\, p,q\in \MM_N\} \]
and such that $\phi$ is fibered if and only if it pairs maximally with a marked vertex of $\MM_N$. 
\end{theorem}

Subsequently,  by a  $(2,1)$--presentation we mean a group presentation with precisely two generators and one non-empty relator. A $(2,1)$--presentation is \emph{cyclically reduced} if the relator is a cyclically reduced word.
Recently the first and the third author  \cite{FT15} associated to a cyclically reduced $(2,1)$--presentation $\pi=\ll x,y\,|\, r\rr$ 
a marked polytope $\MM_\pi$ in $H_1(\pi;\R)$.

Now we  outline the definition of $\MM_\pi$ in the case that $b_1(\pi)=2$. A different (but equivalent) definition is given in Section \ref{section:defpolytopepi}, as well as a definition for cyclically reduced $(2,1)$--presentations  $\pi$ with $b_1(\pi)=1$.

Identify $H_1(G_\pi;\Z)$ with $\Z^2$ such that $x$ corresponds to $(1,0)$ and $y$ corresponds to $(0,1)$. 
Then the relator $r$ determines a discrete walk on the integer lattice in $H_1(G_\pi;\R),$ and the marked polytope $\MM_\pi$ is obtained from the convex hull of the trace of this walk as follows:
\bn
\item Start at the origin and walk across $\Z^2$ reading the word $r$ from the left.
\item Take the convex hull $\mathcal{C}$ of the set of all lattice points reached by the walk.
\item Mark precisely those vertices of $\mathcal{C}$ which the walk passes through exactly once.
\item Now consider the unit squares that are completely contained in $\mathcal{C}$ and touch a vertex of $\mathcal{C}$. Mark a midpoint of a square precisely when one (and hence all) vertices of $\mathcal{C}$ incident with the square are marked.
\item The set of vertices of $\MM_\pi$ is the set of midpoints of all of these squares, and a vertex of $\MM_\pi$ is marked precisely when it is a marked midpoint of a square.
\en

In Figure \ref{fig:dunfield-intro} we sketch the construction of $\MM_\pi$ for the presentation $\pi=\ll x,y\,|\, r\rr,$ where
\[r=x^2yx^{-1}yx^2yx^{-1}y^{-3}x^{-1}yx^2yx^{-1}yxy^{-1}x^{-2}y^{-1}xy^{-1}x^{-2}y^{-1}xy^3xy^{-1}x^{-2}y^{-1}xy^{-1}x^{-1}y.\]
This example is due to Dunfield~\cite{Du01} and presents the fundamental group of the exterior of the  2--component link in $S^3$ shown in Figure \ref{fig:dunfield} (see Section~\ref{section:dunfield}). 
\begin{figure}[h]
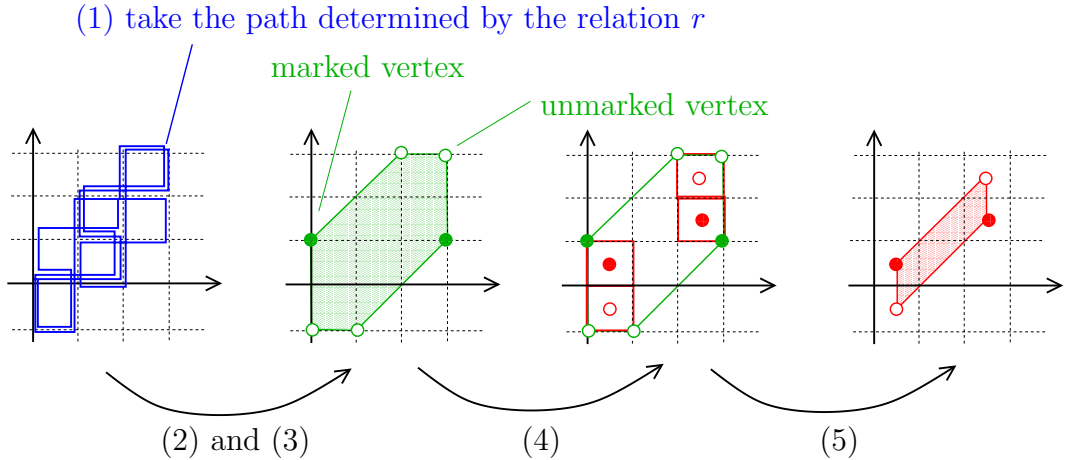
\caption{The marked polytope $\MM_\pi$ for Dunfield's example (labels on arrows correspond to steps in the above algorithm).}\label{fig:dunfield-intro}
\end{figure}

Given two polytopes $\PP$ and $\QQ$ in a vector space $V,$ we write $\PP\doteq \QQ$ if the polytopes $\PP$ and $\QQ$ differ by a translation, i.e. if there exists $v\in V$ with $\PP=v+\QQ$. The following is the main theorem of this paper.

\begin{restatable}{theorem}{mainthm}
\label{mainthm}
Let $N$ be an irreducible  3--manifold that admits a cyclically reduced $(2,1)$--presentation  $\pi=\langle x,y\,|\, r\rangle$. Then
\[ \mathcal{M}_N\doteq \mathcal{M}_\pi.\]
\end{restatable}

The proof of Theorem \ref{mainthm} relies on the Virtually Special Theorem of Agol \cite{Ag13}, Liu \cite{Li13}, Przytycki-Wise \cite{PW12,PW14} and Wise \cite{Wi09,Wi12a,Wi12b}, which we recall in Section \ref{section:special}. It also hinges on the following general result, which is of independent interest. 

\begin{restatable}{theorem}{pinresg}
\label{thm:pinresg}
Let $N$ be an irreducible 3--manifold with empty or toroidal boundary. If $N$ is not a closed graph manifold, then $\pi_1(N)$ is residually a torsion-free and elementary amenable group. 
\end{restatable}

The proof of Theorem~\ref{thm:pinresg} uses the Virtually Special Theorem and builds on work of Linnell--Schick \cite{LS07}. It is proved in Section~\ref{section:3mfdgroups}, where we also give several consequences.\\

We give a brief outline of the proof of Theorem \ref{mainthm}.
The starting point is an alternative definition of the marked polytope $\mathcal{M}_\pi$ using Fox derivatives \cite{Fo53} (see Section~\ref{section:pipolytope}). This definition is less pictorial, but it allows us  to relate the polytope $\mathcal{M}_\pi$ to the chain complex of the universal cover of the 2--complex $X$ associated to the presentation $\pi$. This makes it possible to study the `size' of $\mathcal{M}_\pi$ using twisted Reidemeister torsions corresponding to finite-dimensional complex representations and corresponding to skew fields of $X$ \cite{Wa94,Co04,Ha05,Fr07,FV10}. Since $X$ is simple homotopy equivalent to $N,$ these twisted Reidemeister torsions agree with the twisted Reidemeister torsions of $N$ . 

In the following we denote $\mathcal{P}_N$ and $\PP_\pi$ the polytopes $\MM_N$ and $\MM_\pi$ without the markings.
Given two polytopes $\PP$ and $\QQ$ in a vector space $V$ we write $\PP\leq \QQ$ if there exists $v\in V$ with $v+\PP\subset \QQ$. 
The proof of Theorem \ref{mainthm} now breaks up into three parts:
\begin{enumerate}
\item We first show that $\PP_N\leq \PP_\pi$. Put differently, we  show that $\PP_\pi$ is `big enough' to contain $\PP_N$. This is achieved with the main theorem of \cite{FV12}, which states that twisted Reidemeister torsions corresponding to finite-dimensional complex representations detect the Thurston norm of $N$. 
This relies on the Virtually Special Theorem. See Section~\ref{section:thurstonsmaller}.
\item Next we show the reverse inclusion  $\PP_\pi\leq \PP_N$. This means that   $\PP_\pi$ is `not bigger than necessary'. At this stage it is crucial that $r$ is  cyclically reduced. By \cite{We72} this implies that the summands in the Fox derivative $\frac{\partial r}{\partial x}$ are pairwise distinct elements in the group ring $\Z[\pi]$. Using Theorem \ref{thm:pinresg} and the non-commutative Reidemeister torsions of \cite{Co04,Ha05,Fr07} we show that indeed  $\PP_\pi\subset \PP_N$. See Section~\ref{section:thurstonlarger}.
\item Finally we need to show that the markings of $\MM_N$ and $\MM_\pi$ agree. This follows immediately from \cite[Theorem~1.1]{FT15} and \cite[Theorem~E]{BNS87}. See Section~\ref{section:proofmainthm}.
\end{enumerate}
The paper is concluded with a conjecture and a question in Section \ref{section:questions}.

\subsection*{Convention.}
Throughout this paper, all groups are finitely generated,  all vector spaces finite dimensional, and 
all 3--manifolds are compact, connected and orientable.

\subsection*{Acknowledgments.} 
The first author is very grateful for being hosted by the University of Sydney in March 2013 and again in March 2014. The first and the second author also wish to thank the Mathematische Forschungsinstitut Oberwolfach for hosting them in January 2015.
The first author was supported by the SFB 1085 `Higher invariants', funded by the Deutsche Forschungsgemeinschaft (DFG).  The third author is partially supported under the Australian Research Council's Discovery funding scheme (project number DP140100158). 
The authors thank Nathan Dunfield for very insightful conversations.

%==================================================================
\section{Polytopes associated to 3--manifolds and groups}\label{section:prelims}

%==================================================================
\subsection{Polytopes}
Let $V$ be a real vector space and let $Q=\{Q_1,\dots,Q_k\}\subset V$ be a finite (possibly empty) subset. Denote
\[ \PP(Q)=\hull(Q)=\left\{ \sum_{i=1}^k t_iQ_i\,\left|\ \sum_{i=1}^k t_i= 1, \ t_i\geq 0\right.\right\}.\]
the \emph{polytope spanned by $Q$}.
A \emph{polytope} in $V$ is a subset of the form $\PP(Q)$ for some finite subset $Q$ of $V$.  For any polytope $\PP$ there exists a unique smallest subset $\VV(\PP)\subset \PP$ such that $\PP$ is the polytope spanned by $\VV(\PP)$.
The elements of $\VV(\PP)$ are the \emph{vertices} of $\PP$. 
Note $v\in \PP$ is a vertex if and only if there exists a homomorphism $\phi\co V\to \R$ such that 
$\phi(v)>\phi(p)$ for every $p\ne v \in \PP$.

Let $V$ be a real vector space and let $\PP$ and $\QQ$ be two polytopes in $V$. The \emph{Minkowski sum of $\PP$ and $\QQ$}
is
\[ \PP+\QQ:=\{ p+q\,|\, p\in \PP\mbox{ and }q\in \QQ\}.\]
It is straightforward to see that $\PP+\QQ$ is again a polytope. Furthermore, for each vertex $u$ of $\PP+\QQ$
there exists a unique vertex $v$ of $\PP$ and a unique vertex $w$ of $\QQ$ such that $u=v+w$. Conversely, for each vertex $v$ of $\PP$ there exists a (not necessarily unique) vertex $w$ of $\QQ$ such that $v+w$ is a vertex of $\PP+\QQ$.

If $\PP,\QQ$ and $\RR$ are polytopes with $\PP+\QQ=\RR,$ then we  write $\PP=\RR-\QQ$.
We have
\[ \PP=\{p\in V\,|\, p+\QQ\subset \RR\},\]
in particular $\RR-\QQ$ is well-defined.

There is a natural scaling operation on polytopes:
\[\l\cdot \PP:=\{\l p|p\in \PP\},\]
where $\PP\subset V$ is a polytope and $\l\in \R^+.$ If $k\in \N$, then  the Minkowski sum of $k$ copies of $\PP$ equals $k\PP$.

%==================================================================
\subsection{Convex sets and seminorms}\label{section:convex}
Let $\CC$ be a non-empty convex set in the real vector space $V$. Given $\phi\in \hom(V,\R)$ we define the \emph{thickness of $\CC$ in the $\phi$-direction} by
\[ \th_{\CC}(\phi):=\max\{ \phi(c)-\phi(d)\,|\,c,d\in \CC\}.\]
It is straightforward to see that the function 
\[ \ba{rcl} \l_{\CC}\co \hom(V,\R)&\to& \R_{\geq 0} \\
\phi&\mapsto &\max\{ \phi(c)-\phi(d)\,|\,c,d\in \CC\}\ea \]
is a seminorm.
Conversely, a seminorm $\l\co \hom(V,\R)\to \R_{\geq 0}$ defines the convex set
\[ \CC(\l):=\{v\in V\,|\, \l(v)\leq 1\}.\]
Note that $\CC(\l)$ is \emph{symmetric} since $v\in \CC(\l)$ implies $-v\in \CC(\l)$.
For any seminorm $\l$ we have $\l_{\CC(\l)}=\l$.
On the other hand, if $\CC$ is a non-empty convex set, then $\CC(\l_\CC)$ equals the symmetrization $\CC^{\sym}$ of $\CC$,
\[ \CC^{\sym}:=\{ \tmfrac{1}{2}(c-d)\,|\, c,d\in \CC\}.\]
Finally,  given a convex set $\CC$ in $V$ the \emph{dual} of $\CC$ is
\[ \CC^*:=\{\phi\in \hom(V,\R)\,|\, \phi(v)\leq 1\mbox{ for all }v\in \CC\}.\]

%==================================================================
\subsection{Marked polytopes}
Let $V$ be a real vector space.
A \emph{marked polytope} $\MM$ in $V$ is a polytope $\PP$ and a (possibly empty) subset $\VV^+$ of $\VV(\PP).$ The elements of $\VV^+$ are the \emph{marked vertices}, the elements of $\VV(\PP)\setminus \VV^+$ are the \emph{unmarked vertices} and $\PP$ is the \emph{underlying polytope} of $\MM$.

If $\MM = (\PP, \VV^+)$ and $\NN= (\QQ, \WW^+)$  are two marked polytopes, then the \emph{Minkowski sum of $\MM$ and $\NN$} has underlying polytope the Minkowski sum of the underlying polytopes and set of marked vertices precisely those that are sums of marked vertices:
\[ \MM + \NN = (\;\PP+\QQ, \;\VV(\PP+\QQ) \cap (\VV^++\WW^+)\;).\]
The marked polytope $\MM = (\PP, \VV^+)$ is \emph{symmetric} if the underlying polytope $\PP$ is symmetric and $\VV^+ = - \VV^+.$

%==================================================================
\subsection{The Thurston norm and fibered classes} \label{section:thurstonnorm}
Let $N$ be a 3--manifold. For each $\phi\in H^1(N;\mathbb{Z})$ there is a properly embedded oriented surface $\Sigma$, such that $[\Sigma]\in H_2(N, \partial N; \mathbb{Z})$ is the Poincar\'e dual to $\phi.$ Letting 
$\chi_-(\Sigma)=\sum_{i=1}^k \max\{-\chi(\Sigma_i),0\}$, 
where $\Sigma_1, \ldots, \Sigma_k$ are the connected components of $\Sigma,$ the \emph{Thurston norm} of $\phi\in H^1(N;\mathbb{Z})$  is 
 \[
x_N(\phi)=\min \big\{ \chi_-(\Sigma) \,|\, [\Sigma] = \phi \big\}.
\]

The class $\phi\in H^1(N;\R)$ is called \emph{fibered} if it can be represented by a non-degenerate closed 1-form.
By \cite{Ti70} an integral class $\phi\in H^1(N;\Z)=\hom(\pi_1(N),\Z)$ is fibered if and only if there exists a fibration
$p\co N\to S^1$ such that $p_*=\phi:\pi_1(N)\to \pi_1(S^1)=\Z$.

Thurston \cite{Th86} showed that $x_N$ extends to a seminorm $x_N$ on $H^1(N;\R)$ and that the dual $\CC(x_N)^*$ to the unit norm ball $\CC(x_N)$ of the seminorm $x_N$ 
is a polytope $\PP_N$ with vertices in $\im\{H_1(N;\Z)/\mbox{torsion}\to H_1(N;\R)\}$. Furthermore, Thurston showed that we can turn $\PP_N$ into a marked polytope $\MM_N,$ which has the property that $\phi\in H^1(N;\R)=\hom(H_1(N;\R),\R)$ is fibered if and only if it pairs maximally with a marked vertex.

%==================================================================
\subsection{The marked polytope for elements of group rings}
Let $G$ be a group.
Throughout this paper, given $f\in \C[G]$ and $g\in G$ we denote $f_g$ the $g$--coefficient of $f$. Let $\psi\co G\to H_1(G;\Z)/\mbox{torsion}$ be the canonical map.

We write $V=H_1(G;\R)$ and we view
$H_1(G;\Z)/\mbox{torsion}$ as a subset of $V$. With this convention the above map $\psi$ gives rise to a map  $\psi\co G\to V$.
Given $f\in \C[G]$ we  refer to 
\[ \PP(f):=\PP\left(\{\psi(g)\,|\, g\in G\mbox{ with } f_g\ne 0\}\right)\subset V\]
as the \emph{polytope of $f$}.
We will now associate to $\PP(f)$ a marking. In order to do this we need a few more definitions.
\bn
\item 
For $v\in V$ we refer to $ f^v:=\sum_{g\in \psi^{-1}(v)}f_gg $
as the \emph{$v$--component of $f$}.
\item We say that an element $r\in \C[G]$ is a monomial if it is of the form  $r=\pm g$ for some $g\in G$.
\en
A vertex $v$ of $\PP(f)$ is \emph{marked} precisely when the $v$-component of $f$ is a  monomial.
We then refer to the polytope $\PP(f)$ together with the set of all marked vertices 
as the \emph{marked polytope $\MM(f)$ of $f$}. 
The proof of \cite[Lemma 3.2]{FT15} applies with the above definitions, to give:

\begin{lemma}\label{lem:productadd}
Let $G$ be a group and let $f,g\in \C[G]$. Then the following hold:
\bn
\item If for every vertex $v$ of $\PP(f)$ the $v$-component $f^v\in \C[G]$ is  not a zero divisor, then $\PP(f\cdot g)=\PP(f)+\PP(g)$.
\item If each vertex  of $\MM(f)$ is marked, then $\MM(f\cdot g)=\MM(f)+\MM(g)$.
\en 
\end{lemma}

%==================================================================
\subsection{The marked polytope for a  $(2,1)$--presentation}\label{section:pipolytope}
\label{section:defpolytopepi}

Let $F$ be the free group with generators $x$ and $y$.
Following \cite{Fo53} we  denote by  $\frac{\partial }{\partial x}\co \Z[F]\to \Z[F]$ the \emph{Fox derivative with respect to $x$}, i.e. the unique $\Z$--linear map such that
\[
\frac{\partial 1}{\partial x}=0,\quad \frac{\partial x}{\partial x}=1, \quad \frac{\partial y}{\partial x}=0\;\mbox{ and }\;
\frac{\partial uv}{\partial x}=\frac{\partial u}{\partial x}+u\frac{\partial v}{\partial x}\]
for all  $u,v\in F$. We similarly define the Fox derivative with respect to $y$, and often write
\[ u_x:=\frac{\partial u}{\partial x}\;\mbox{ and }\;u_y:=\frac{\partial u}{\partial y}.\]

In \cite{FT15} we proved the following proposition.

\begin{proposition}\label{prop:foxpolytope}
Let $\pi=\ll x,y\,|\,r \rr$
be a  $(2,1)$--presentation with $b_1(\pi)=2$. Then there exists a marked polytope $\MM$, unique up to translation, such that 
\[ \MM+\MM(x-1)\doteq \MM(r_y) \;\mbox{ and }\; \MM+\MM(y-1)\doteq \MM(r_x).\]
\end{proposition}

Denote $\MM_\pi$ the  marked polytope of Proposition~\ref{prop:foxpolytope}. Up to translation it is a well-defined invariant of the presentation, and it is shown in \cite{FT15} that this definition is equivalent to the one sketched in the introduction. 
 
A $(2,1)$--presentation $\pi=\ll x,y\,|\,r \rr$ is \emph{simple} if $b_1(G_\pi)=1$, $x$ defines a generator of $H_1(\pi;\Z)/\mbox{torsion}$ and $y$ represents the trivial element in $H_1(\pi;\Z)/\mbox{torsion}$. In \cite{FT15} we showed that given a simple $(2,1)$--presentation $\pi=\ll x,y\,|\,r \rr$  there exists a marked polytope $\MM_\pi$, unique up to translation, such that 
\[ \MM_\pi+\MM(x-1)\doteq \MM(r_y).\]
It was shown in \cite{FT15} that there is a canonical way to associate to
any $(2,1)$--presentation $\pi=\ll x,y\,|\,r \rr$ with $b_1(G_\pi)=1$ a simple presentation $\pi'=\ll x',y'\,|\,r'\rr$ representing the same group.
We then define $\MM_\pi:=\MM_{\pi'}$. 
 
%==================================================================
\subsection{3--manifold groups which admit $\mathbf{(2,1)}$--presentations}
\label{section:3manifoldgroups}
Manifolds having fundamental group with a $(2,1)$--presentation are described in Section~\ref{sec:Examples}. The only specific result needed to develop our theory is the following, which follows from work of Epstein \cite{Ep61}.

\begin{theorem}\label{thm:ep61}
Let $N$ be an irreducible (compact, connected and orientable) 3--manifold such that $\pi:=\pi_1(N)$ admits a  $(2,1)$--presentation.
Then the  boundary of $N$ consists of one or two tori.
\end{theorem}

\begin{proof}
It follows from \cite[Lemma 1.7]{Ep61} that $\pi$ has deficiency 1, and from \cite[Section~3]{Ep61} that the fundamental group of a closed irreducible 3--manifold has deficiency zero. Whence $N$ has non-empty boundary, and \cite[Lemma 2.2]{Ep61} implies that $\frac{1}{2} \chi(\partial N) = \chi(N)\ge 0.$ The 3--ball is the only irreducible 3--manifold with a spherical boundary component due to the Poincar\'e conjecture. Hence no boundary component of $N$ is a sphere since $\pi_1(N)\ne \{1\}$. Since $N$ (and hence each of its boundary components) is orientable, we now have $\chi(\partial N) = 0$ and every boundary component is a torus.

A standard half-lives-half-dies argument shows that $b_1(\partial N)\leq 2b_1(N)$. Since $b_1(N)\leq 2$ we deduce that $\partial N$ either consists of one or two tori.
\end{proof}

%==================================================================
\section{Properties of 3--manifold groups}
\label{section:3mfdgroups}

%==================================================================
\subsection{The Virtually Special Theorem}
\label{section:special}
Following Agol \cite{Ag08} we say that a group is  \emph{residually finite rationally solvable \textup{(}RFRS\textup{)}} if there
exists a sequence~$\{\pi_n\}$  of subgroups of $\pi$
such that  $\bigcap_n \pi_n=\{1\}$ and such that for any $n$ the following hold:
\bn
\item  the subgroup $\pi_n$ is  normal and finite index in $\pi$,
\item  $\pi_n\supseteq\pi_{n+1}$, and
\item the natural surjection $\pi_n\to \pi_n/\pi_{n+1}$ factors through $\pi_n\to H_1(\pi_n;\Z)/\mbox{torsion}$.
\en

As usual, given a property of groups or spaces we say this property is satisfied \emph{virtually} if a finite-index subgroup (not necessarily normal) or a finite-index cover (not necessarily regular) has the property.

The following theorem is now a variation on the Virtually Special Theorem.

\begin{theorem}\label{thm:virtuallyrfrs}
Let $N$ be an irreducible 3--manifold with empty or toroidal boundary. If $N$ is not a closed graph manifold, then $\pi_1(N)$ is virtually RFRS.
\end{theorem}

The theorem was proved by Agol \cite{Ag13} for all closed hyperbolic 3--manifolds, by Wise \cite{Wi09,Wi12a,Wi12b} for all hyperbolic 3--manifolds with boundary, by Liu \cite{Li13} and Przytycki--Wise \cite{PW14} for all graph manifolds with boundary and by Przytycki--Wise \cite{PW12} for all 3--manifolds with a non-trivial JSJ-decomposition that has at least one hyperbolic JSJ-component. We refer to \cite{AFW15} for precise references.

In the following, given a 3--manifold $N$ we say that $\phi\in H^1(N;\R)$ is \emph{quasi-fibered} if it is a limit of fibered classes in $H^1(N;\R)$. 
The following theorem is an immediate consequence of Theorem \ref{thm:virtuallyrfrs} and Agol's virtual fibering theorem \cite[Theorem~5.1]{Ag08} (see also~\cite[Theorem~5.1]{FKt14} for an exposition).

\begin{theorem}\label{thm:virtuallyfib}\label{thm:virtualfib}
Let $N$ be an irreducible 3--manifold with empty or toroidal boundary. If $N$ is not a closed graph manifold, then for every $\phi \in H^1(N;\R)$ there exists a finite-index cover $p\colon N'\to N$ such that $p^*(\phi)$ is quasi-fibered.
\end{theorem}

If we apply the theorem to the zero class we get in particular the following corollary.

\begin{corollary}\label{cor:virtuallyfib}
An irreducible 3--manifold with empty or toroidal boundary is virtually fibered unless it is a closed graph manifold.
\end{corollary}

%==================================================================
\subsection{Residual properties of 3--manifold groups}
\label{section:resg}
We start with several definitions, most of which are standard. Let $\PP$ be a class of groups.
\bn
\item The group  $\pi$ is \emph{residually $\PP$} if for every non-trivial $g\in\pi$, there exists a homomorphism $\a\colon \pi\to \G$ to a group in $\G$ in $\PP$  such that $\a(g)\neq 1$. 
\item The group $\pi$ is \emph{fully residually $\PP$} if for every finite subset  $\{g_1,\dots,g_n\}\subset\pi\setminus \{1\}$, there exists a epimorphism $\a\colon \pi\to G$ to a group in $\G$ in $\PP$  such that $\a(g_i)\neq 1$ for all $i=1,\dots,n$. 
\item 
The group $\pi$ has the \emph{$\PP$--factorization property} if for every epimorphism $\a\colon \pi\to G$ onto a finite group $G$ there exists an epimorphism $\b\colon \pi\to \G$ to a group $\G$ in $\PP$  such that $\a$ factors through $\b$. 
\en
We are mostly interested in the following classes of groups.
\bn
\item The class $\EEAA$ of \emph{elementary amenable groups} is  the  smallest  class  of  groups \index{group!elementary-amenable}
that  contains  all  abelian  and  all  finite  groups and that  is closed under extensions and directed  unions. 
\item We denote by $\TTEEAA$ the class of all groups that are torsion-free and elementary amenable. It is clear that $\TTEEAA$ is closed under taking finite direct products.
\en
 
Using  Corollary \ref{cor:virtuallyfib} and work of Linnell--Schick \cite{LS07} we will prove the  following theorem:

\begin{theorem}\label{thm:pi1nfactors}
Let $N$ be an irreducible 3--manifold with empty or toroidal boundary. If $N$ is not a closed graph manifold, then $\pi_1(N)$ has the $\TTEEAA$-factorization property.
\end{theorem}

The question of to what degree this statement holds for closed graph manifolds is discussed in Section~\ref{section:closedgraphmanifolds}. We postpone the proof of the theorem to Section \ref{section:proofpi1}, and point out several consequences, including the following, which was stated in the introduction.

\pinresg*

\begin{proof}
Let $\PP$ be any class of groups. If  a group $\pi$ is residually finite and has the $\PP$--factorization property, then $G$ is also residually $\PP$. The statement of the theorem now follows from Theorem \ref{thm:pi1nfactors} and the fact that 3--manifold groups are residually finite \cite{He87}.
\end{proof}

\begin{corollary}\label{cor:groupring}
Let $\pi$ be the fundamental group of an irreducible 3--manifold that has empty or toroidal boundary and is not a closed graph manifold. For every non-zero element $p\in \Z[\pi],$ there exists a homomorphism $\a\colon \pi\to \G\in \TTEEAA$ such that $0 \neq \a(p)\in \Z[\Gamma]$. 
\end{corollary}

\begin{proof}
We write $p=\sum_{i=1}^ka_ig_i$, where $a_1,\dots,a_k\ne 0$ and $g_1,\dots,g_n\in \pi$ are pairwise distinct. By Theorem \ref{thm:pinresg} the group $\pi$ is residually $\TTEEAA$. Since $\TTEEAA$ is closed under taking finite direct products, $\pi$ is also fully residually $\TTEEAA$.
We can thus find a homomorphism $\a\colon \pi\to \G$ to a group $\G\in \TTEEAA$
such that all $\a(g_i)$ and all products $\a(g_ig_j^{-1}),$ $i\ne j,$ are non-trivial. Whence $\a(p)\in \Z[\G]$ is non-zero.
\end{proof}

%==================================================================
\subsection{Proof of Theorem \ref{thm:pi1nfactors}}\label{section:proofpi1}
The following lemma is probably well-known to the experts.

\begin{lemma}\label{lem:tf}
Let $E$ be a surface group $($i.e.\ the fundamental group of a compact orientable surface$)$ and let $R\subset E$ be a normal subgroup. Then $E/[R,R]$ is torsion--free.
\end{lemma}

\begin{proof}
Let $g\in E/[R,R]$ be a non--trivial element. We pick a representative for $g$ in $E$ which by slight abuse of notation we also denote by $g$.
We denote by $S$ the subgroup of $E$ generated by $g$ and $R$. It suffices to prove the following claim.

\begin{claim}
The group $S/[R,R]$ is torsion--free.
\end{claim}

 We  consider the following short exact sequence
\[ 1\to [S,S]/[R,R]\to S/[R,R]\to S/[S,S]\to 0.\]
Since $R$ and $S$ are either surface groups or free we deduce that $S/[S,S]=H_1(S;\Z)$ and $R/[R,R]=H_1(R;\Z)$ are torsion--free. 
The group  $S/R$ is generated by one element, which implies that $S/R$ is cyclic, in particular abelian. It follows that $[S,S]\subset R$. We thus see that  $[S,S]/[R,R]$ is a subgroup of $R/[R,R]$.
So the groups on the left and on the right of the above short exact sequence are torsion-free. It follows that $S/[R,R]$ is torsion--free. 
\end{proof}

 \begin{proposition}\label{thm:ext}
If $1 \to E \to \pi \to M \to 1$ is an exact sequence with $E$ a surface group and $M \in \TTEEAA$, then $\pi$ has the $\TTEEAA$-factorization property.
\end{proposition}
\begin{proof}
 Let $\alpha\colon \pi \to P$ be a map to a finite group. 
Let $R = E \cap \ker \alpha$.
By Lemma \ref{lem:tf} the group $E / [R,R]$ is torsion-free.
Furthermore it is  elementary amenable by the exact sequence 
\[1 \rightarrow R/ [R,R] \rightarrow E / [R,R] \rightarrow E/R \rightarrow 1.\]
Now $\alpha$ factors through $\pi/ [R,R],$ and this is in $\TTEEAA$ due to the sequence 
\[1 \to E / [R,R] \to \pi/ [R,R] \to M \to 1.\hfill\qedhere\]
\end{proof}

The \emph{profinite completion} of the group $\pi$ is denoted $\widehat{\pi};$ see \cite[Section~3.2]{RZ10} for a definition and its main properties. 
Following Serre \cite[D.2.6~Exercise~2]{Se97} we say that a group $\pi$ is \emph{good} if the natural morphism $H^*(\widehat{\pi};A)\to H^*(\pi;A)$
is an isomorphism for any finite abelian group $A$ with a $\pi$-action. 

\begin{theorem}\label{thm:extendfactorizingproperty}
Let $\pi$ be a finitely generated  torsion-free group that has a finite-dimensional classifying space and which is good.
If $\pi$  admits a finite index subgroup $\G$ which has the $\TTEEAA$-factorization property, then $\pi$ also has the $\TTEEAA$-factorization property.
\end{theorem}

In the proof of the theorem we will on several occasions use the following standard notation: if $\G$ is a subgroup of $\pi$, then $\G^\pi:=\cap_{g\in \pi} g\Gamma g^{-1}$. Note that $\G^\pi$ is always a normal subgroup of $\pi$, and if $\G$ is of finite-index, then $\G^\pi$ is of finite-index. We also note that the methods of the proof build heavily on the work of Linnell and Schick \cite{LS07}.

\begin{proof}
Let $\a\colon \pi\to G$ be a homomorphism to a finite group. We denote by $K\subset \pi$ the intersection of $\ker(\a)$ and $\G^\pi$.
The subgroup $K$ is of finite index in $\pi$ and is clearly contained in $\Gamma$. It is straightforward to see that $K$ also has the $\TTEEAA$-factorization property. We write $Q:=\pi/K$. First suppose that $Q$ is a $p$-group. It suffices to show there is a subgroup $U \trianglelefteq \pi$ such that the map $\pi \rightarrow Q$ factors through $\pi/U$ and $\pi/U$ is in $\TTEEAA$.

If no such $U$ exists, then since $K$ has the $\TTEEAA$-factorization property, there is a nontrivial subgroup $Q'$ of $Q$ that splits in the induced sequence of profinite completions $$1 \to \hat K \rightarrow \hat \pi \rightarrow Q \rightarrow 1,$$ see \cite[Lemmas 6.7, 6.8, 6.9]{Sc14}. However, putting the following two observations together shows that this is not possible.
\begin{enumerate}
\item The cohomology $H^*(Q',\F_p)$ is nonzero in infinitely many dimensions.
\item By \cite[Exercise 1, 2.6]{Se97} any finite-index subgroup $L$ (such as $K$ or the preimage of $Q'$ under $\pi\to \Q$) of $\pi$ is also good and has a finite-dimensional classifying space. This implies that $H^*(\hat L,\F_p) \cong H^*(L,\F_p)$ is nonzero in only finitely many dimensions.
\end{enumerate}

For the general case, we use a trick from \cite{LS07}. For each Sylow $p$-subgroup $S$ of $Q$, consider the exact sequence $1\to K \to \pi_S\to S\to 1$, where $\pi_S$ is the preimage of $S$. By the above, we get a collection of subgroups $U_{\pi_S}$ such that the quotients $\pi_S/U_{\pi_S}$ are torsion-free elementary amenable. Let $U = \cap_S U_{\pi_S}$. Since $\pi/U^\pi$ is a finite extension of $\G/U^\pi$, elementary amenability follows again from \cite[Lemma~4.11]{LS07}. It remains to show that $\pi/U^\pi$ is torsion-free.

Suppose that  $\pi/U^\pi$ has a non-trivial torsion element $\gamma$. By raising $\gamma$ to some power we get an element $\gamma'$ that is $p$-torsion for some prime $p$. We have an exact sequence $$1 \to U_S^\pi/U^\pi \to \pi_S/U^\pi \to \pi_S/U_S^\pi \to 1$$ where $U_S^\pi/U^\pi$ and $\pi_S/U_S^\pi$ are torsion-free by Lemma 4.11 in \cite{LS07}. Since $K/U^\pi$ is torsion-free, $\gamma'$ would map to some Sylow $p$-subgroup, in which case $\gamma' \in \pi_S/U_S^\pi$, which is torsion-free. Therefore, $\pi/U^\pi$ is torsion-free. 
\end{proof}

Now we are finally in a position to prove Theorem~\ref{thm:pi1nfactors}.

\begin{proof}[Proof of Theorem~\ref{thm:pi1nfactors}]
Let $N$ be an irreducible 3--manifold that has empty or toroidal boundary and that is not a closed graph manifold. According to Corollary \ref{cor:virtuallyfib}, $N$ has a finite cover $M$ that is fibered. 
The fundamental group of $M$ is a semidirect product of $\Z$ with a surface group, and hence Lemma~\ref{lem:tf} and Proposition~\ref{thm:ext} imply $\pi_1(M)$ has  the $\TTEEAA$-factorization property.

It follows from  \cite[Exercise 2b) pg. 16]{Se97} that $\pi_1(M)$ is good.
By  \cite[Exercise 1 pg. 16]{Se97}  the group $\pi_1(N)$  is also good. 
It is well-known, see e.g.\ \cite[(A1)]{AFW15}, that $N$ is aspherical and that in particular $\pi_1(N)$ is torsion-free.
Thus we can apply Theorem~\ref{thm:extendfactorizingproperty} to $\pi_1(N)$ and the finite-index subgroup $\pi_1(M)$, giving the desired result that 
 $\pi_1(N)$ has  the $\TTEEAA$-factorization property.
\end{proof}

%==================================================================
\subsection{The case of closed graph manifolds}\label{section:closedgraphmanifolds}
It is natural to ask for which closed graph manifolds the conclusions of  Theorem~\ref{thm:pi1nfactors} and its corollaries hold.
It follows from the work of Liu~\cite{Li13} that the conclusion of the theorem also holds for closed non-positively curved graph manifolds.
The question of which closed graph manifolds are non-positively curved was treated in detail by Buyalo and Svetlov~\cite{BS05}.
In the following we give a short list of examples of graph manifolds that are not non-positively curved:
\bn
\item spherical 3-manifolds,
\item Sol- and Nil-manifolds,
\item Seifert fibered 3-manifolds that are finitely covered by a non-trivial $S^1$-bundles over a closed surface.
\en
It is clear that the statements do not hold for spherical 3-manifolds with non-trivial fundamental group. The following lemma takes care of the second case.

\begin{lemma}
The fundamental groups of  Sol- and Nil-manifolds are $\TTEEAA$, in particular they have the
 $\TTEEAA$-factorization property.
\end{lemma}

\begin{proof}
Sol- and Nil-manifolds are finitely covered by torus-bundles over $S^1$. Hence their fundamental groups are elementary amenable, but the fundamental groups are also torsion-free, so they are $\TTEEAA$. 
\end{proof}

\begin{lemma}
Let $N$ be a Seifert fibered space with infinite fundamental group.  Then $\pi_1(N)$ has the  $\TTEEAA$-factorization property.
\end{lemma} 

\begin{proof} Since we will not make use of this lemma we only sketch the proof. The manifold $N$ is finitely covered by an $S^1$-bundle  over a surface. By Theorem~\ref{thm:extendfactorizingproperty} we can thus without loss of generality assume that $N$ is an $S^1$-bundle over a surface $F$.
Since $\pi_1({N})$ is infinite there exists a short exact sequence
\[ 1\to \ll t\rr \to \pi_1(N)\to \pi_1(F)\to 1,\]
where the subgroup $\ll t\rr$ is generated by the $S^1$-fiber.
By Proposition~\ref{thm:ext} the group $\pi_1(F)$ has the  $\TTEEAA$-factorization property. Denote by $e$  the Euler number of the $S^1$-bundle over $F$ and denote by $M$ the total space of the $S^1$-bundle over the torus with Euler number $e$. Then there exists a fiber-preserving map from ${N}$ to $M$. Since $\pi_1(M)$ is $\TTEEAA$ we found a homomorphism from $\pi_1(N)$ to a $\TTEEAA$ group which is injective on $\ll t\rr$. Now it is straightforward to see that $\pi_1(N)$ has the  $\TTEEAA$-factorization property.
\end{proof}

The above discussion shows that the fundamental groups of many closed graph manifolds have the $\TTEEAA$-factorization property. Nonetheless we expect that there are many closed graph manifolds whose fundamental groups do not have the $\TTEEAA$-factorization property.

%==================================================================
\section{Proof of Theorem  \ref{mainthm} (I)}
\label{section:thurstonsmaller}

The goal of this section is to prove the following proposition.

\begin{proposition}\label{mainprop1}
Let  $\pi=\langle x,y\,|\, r\rangle$  be a cyclically reduced $(2,1)$--presentation for the fundamental group of an irreducible  3--manifold $N$.  Then 
\[ \mathcal{P}_N\leq \mathcal{P}_\pi.\]
\end{proposition}

The main ingredient in the proof will be the fact that twisted Reidemeister torsions  corresponding to finite-dimensional complex representations detect the Thurston norm of 3--manifolds.

%==================================================================
\subsection{Tensor representations}
Let $\pi$ be a group, let  $\a \co \pi\to \gl(k,\C)$ be a  representation and let  $\psi\co \pi\to H$ be a homomorphism to a free abelian group.
We denote by $\C(H)$ the quotient field of the group ring $\C[H]$. 
The homomorphisms  $\a$ and $\psi$ give rise to the representation
\[ \begin{array}{rcl} \a \otimes \psi \co \pi &\to & \gl(k,\C(H)) \\
g&\mapsto & \a(g)\cdot \psi(k)\ea \]
that we refer to as the tensor product of $\a$ and $\psi$.
This representation  extends to a ring homomorphism $\Z[\pi]\to M(k\times k,\C(H)),$ which we also denote by $\a\otimes \psi$.

%==================================================================
\subsection{The definition of the twisted Reidemeister torsion}\label{section:twitorsion}

Let $X$ be a finite $CW$--complex, $\pi:=\pi_1(X)$, and denote by $\ti{X}$ the universal cover of $X$. The action of $\pi$ via  deck transformations  on $\ti{X}$ equips the chain complex $C_*(\ti{X};\Z)$ with the structure of a chain complex  of  $\Z[\pi]$--left modules.

Let  $\a \co \pi\to \gl(k,\C)$ be a  representation. We denote by  $\psi\co \pi\to H:=H_1(X;\Z)/\mbox{torsion}$  the obvious projection map. 
Using the representation $\a\otimes \psi$ we can now view $\C(H)^k$ as a right $\Z[\pi]$-module,
where the action is given by right multiplication on row-vectors.

We  consider  the  chain complex
\[ C_*(X;\C(H)^k):=\C(H)^k \otimes_{\Z[\pi]} C_*(\ti{X};\Z) \]
of $\C(H)$--modules. For each cell in $X$ pick a lift to a cell in $\wti{X}$. 
We denote by $e_1,\dots,e_k$ the standard basis for $\C(H)^k$. The tensor products of the lifts of the cells and the vectors
$e_i$ turns $C_*(X;\C(H)^k)$ into a chain complex of based $\C(H)$--vector spaces.

If the chain complex $C_*(X;\C(H)^k)$ is not acyclic, then we define the corresponding twisted Reidemeister torsion $\tau(X,\a)$ to be zero. Otherwise we denote by $\tau(X,\a)\in \C(H)\sm \{0\}$ the torsion of the based chain complex  $C_*(X;\C(H)^k)$.
We refer to \cite{Tu01} for the definition of the torsion of a based chain complex.
Standard arguments show that $\tau(X,\a)\in \C(H)\sm \{0\}$ is well-defined up to multiplication
by an element of the form $zh,$ where $z\in \pm \det(\a(\pi))$ and $h\in H$.
The indeterminacy arises from the fact that we had to choose lifts and an ordering of the cells.

Suppose $N$ is a 3--manifold and let $\a \co \pi_1(N)\to \gl(k,\C)$ be  a  representation.
Choose a $CW$--structure $X$ for $N$ and define $\tau(N,\a):=\tau(X,\a)$. It is well-known, see e.g.\ \cite{Tu01,FV10}, that this definition does not depend on the choice of the $CW$--structure.

%==================================================================
\subsection{The polytopes corresponding to twisted Reidemeister torsion} 
As above, suppose $N$ is a 3--manifold and $\a \co \pi_1(N)\to \gl(k,\C)$ a  representation.
If $\tau(N,\a)$ is zero, then we define $\TT(N,\a)=\emptyset$. 

Otherwise we write $\tau(N,\a)=p\cdot q^{-1}$ with $p,q\in \C[H]$. 
If the Minkowski difference $\PP(q)-\PP(p)$ exists (and by \cite[p.~53]{FV10} this is the case if $b_1(N)\geq 2$), then we define
\[ \TT(N,\a):=\frac{1}{k}\cdot (\PP(p)-\PP(q)),\]
and otherwise define $\TT(N,\a):=\{0\}$.

\begin{proposition}\label{prop:tapcontainedinmpi}
Let $\pi=\ll x,y\, |r\rr $ be a $(2,1)$--presentation  for the fundamental group of an irreducible 3--manifold $N$. Then for any  representation
we have 
\[ \TT(N,\a)\subset \PP_\pi.\]
\end{proposition}

In the proof of the proposition we will need one more definition and one more lemma.
Let $\pi$ be a group, $f\in \Z[\pi],$ $\a\co \pi\to  \gl(k,\C)$ be a representation, and $\psi_\pi\co \pi\to H:=H_1(\pi;\Z)/\mbox{torsion}$ be the canonical epimorphism.  Then $\det((\a\otimes \psi_\pi)(f))\in\C[H]$ and we write
\[ \PP(f,\a):=\frac{1}{k} \PP(\det((\a\otimes \psi_\pi)(f)))\subset H_1(\pi;\R).\]

\begin{lemma}\label{lem:detcontained}
Let $\pi$ be a group, $f\in \Z[\pi]$ and $\a\co \pi\to  \gl(k,\C)$ be  a  representation. Then
\[\PP(f,\a)\subset \PP(f).\]
\end{lemma}

\begin{proof}
We write $f=c_1h_1+\dots+c_lh_l$ with $h_1,\dots,h_l\in \pi$ and $c_1,\dots,c_l\ne 0$.
We consider
\[ S:=\{ s_1\psi(g_1)+\dots+s_l\psi(h_l)\,|\, s_1,\dots,s_l\in \C\}.\]
Put differently, $S$ is the set of all elements in $\C[H]$ with support a subset of $\{\psi(g_1),\dots,\psi(g_l)\}$. For every $p\in S$ we have $\PP(p)\subset \PP(\psi(g_1),\dots,\psi(g_l))= \PP(f)$. 
This implies that if $p_1,\dots,p_k$ are elements in $S$, then \[\PP(p_1\cdots\dots \cdot p_k)=\PP(p_1)+\dots+\PP(p_k)\subset \PP(f)+\dots+\PP(f)=k\PP(f).\]

We write $M:= (\a\otimes \psi)(f)=\sum_{i=1}^l c_i\a(h_i)\cdot \psi(h_i)$. 
Each entry of  $\det(M)$ lies in $S$.
It follows from the Laplace formula that $\det(M)$ is a sum of products of the form $p_1\cdot \dots \cdot p_k,$ where each $p_i$ lies in $S$.
By the above we have $\PP(p_1\cdots\dots \cdot p_k)\subset k\PP(f)$.
The definitions imply that if $a,b\in \C[\pi]$ are such that $\PP(a)$ and $\PP(b)$ are contained in a polytope $\QQ$, then we have also have $\PP(a+b)\subset \QQ$. Hence $\PP(\det(M))\subset k\PP$.
\end{proof}

\begin{proof}[Proof of Proposition \ref{prop:tapcontainedinmpi}]
We again denote by $\psi\co \pi_1(N)\to H_1(N;\Z)/\mbox{torsion}$ the canonical epimorphism.  Note that $\psi(x)\ne 0$ or $\psi(y)\ne 0$. Without loss of generality we may assume $\psi(y)\ne 0$. 

Theorem~\ref{thm:ep61} shows that $N$ has non-trivial toroidal boundary.
It thus follows from \cite[Theorem~A]{Ki96}, see also \cite[p.~50]{FV10},
that 
\[ \tau(N,\a)=\det((\a\otimes \psi)(r_y))\cdot \det((\a\otimes \psi)(y-1))^{-1}.\]
By  Lemma \ref{lem:detcontained} we have $\PP(r_y,\a)\subset \PP(r_y)$. Since $\psi(y)\ne 0$ we know that $\psi(y)$ and $1$ are the two, distinct, vertices of $\PP(y-1)$. Also, we have $\PP(y-1,\a)=\frac{1}{k}\PP(\det(\a(y)\psi(y)-\id_k))$ and it is straightforward to see that this polytope equals $\PP(y-1)$. 

Combining these results we obtain 
\[ \TT(N,\a)=\PP(r_y,\a)-\PP(y-1,\a)=\PP(r_y,\a)-\PP(y-1)\subset \PP(r_y)-\PP(y-1)=\PP_\pi.\qedhere\]
\end{proof}

%==================================================================
\subsection{The proof of Proposition \ref{mainprop1}}
Proposition \ref{mainprop1} is  an immediate consequence of Theorem \ref{thm:ep61},
Proposition \ref{prop:tapcontainedinmpi} and the second statement of the following proposition.

\begin{proposition}\label{prop:tapdetecttn}
Let $N$ be a 3--manifold with empty or toroidal boundary and $\a \co \pi_1(N)\to U(k,\C)$ be  a unitary representation.
Then
\[ \TT(N,\a)\leq \PP_N.\]
Furthermore, if $N$ is  irreducible, then there exists a unitary representation  $\a \co \pi_1(N)\to U(k,\C)$
such that  
\[ \TT(N,\a)\doteq \PP_N.\]
\end{proposition}

\begin{proof}
Let $N$ be a 3--manifold with empty or toroidal boundary. We write $\pi=\pi_1(N)$. Let  $\a \co \pi\to U(k,\C)$ be  a unitary representation. 
If $\tau(N,\a)=0$, then there is nothing to show. So suppose that $\tau(N,\a)\ne 0$.
In \cite[Theorem~1.1]{FKm06} and \cite[Theorem~3.1]{FKm08} it was shown that for any $\phi\in H^1(N;\R)=\hom(\pi,\R)$ we have 
\[ \max\{ \phi(p)-\phi(q)\,|\,p,q\in \TT(N,\a)\}\leq x_N(\phi).\]
It follows from the definitions and the discussion in Section \ref{section:convex} that
$\TT(N,\a)^{\sym}\leq \PP_N$. 
Since $\a$ is a unitary representation it follows from \cite[Theorem~1.2]{FKK12} that $\TT(N,\a)^{\sym}\doteq \TT(N,\a)$. It thus follows that indeed $ \TT(N,\a)\leq \PP_N$.

If $N$ is not a closed graph manifold, then, building on Theorem~\ref{thm:virtualfib}, it was shown  in \cite[Corollary~5.10]{FV12} that there exists a unitary representation  $\a \co \pi\to U(k,\C)$ such that 
\[ \max\{ \phi(p)-\phi(q)\,|\,p,q\in \TT(N,\a)\}= x_N(\phi)\]
for every $\phi\in H^1(N;\R)$. The same argument as above then implies that $ \TT(N,\a)\doteq \PP_N$.
If $N$ is a closed graph manifold, then the same statement holds by \cite{FN15}. 
\end{proof}

%==================================================================
\section{Proof of Theorem  \ref{mainthm} (II)}
\label{section:thurstonlarger}

The goal of this section is to prove the following proposition, and to complete the proof of the main theorem.

\begin{proposition}\label{mainprop2}
Let  $\pi=\langle x,y\,|\, r\rangle$  be a cyclically reduced $(2,1)$--presentation for the fundamental group of an irreducible  3--manifold $N$.  Then 
\[ \mathcal{P}_\pi^{\sym}\leq \mathcal{P}_N.\]
\end{proposition}

In the proof of Proposition~\ref{mainprop1} we used twisted Reidemeister torsions corresponding to finite-dimensional complex representations.
In the proof of Proposition~\ref{mainprop2} we use a different, but related, object, namely non-commutative Reidemeister torsions. In this context they were first studied in \cite{COT03,Co04,Ha05,Fr07}.

%==================================================================
\subsection{The Ore localization of group rings and degrees}\label{section:ore}

Let $\G\in\TTEEAA$.
It follows from \cite[Theorem~1.4]{KLM88} that the group ring $\Z[\G]$ is a domain.
Since  $\G$ is amenable  it follows from \cite[Corollary~6.3]{DLMSY03} that $\Z[\G]$ satisfies the Ore condition. This means that for any
two non-zero elements $x,y\in \Z[\G]$ there exist non-zero elements $p,q\in \Z[\G]$ such that $xp=yq$. By \cite[Section~4.4]{Pa77} this implies that 
$\Z[\G]$ has a classical fraction field, referred to as the Ore localization of $\Z[\G]$, which we denote by $\K(\G)$. 

Let $\phi\co \G\to \Z$ be a homomorphism. For every non-zero $p=\sum_{g\in \G}p_gg\in \Z[\G]$ we define
\[ \deg_\phi(p)=\max\{ \phi(g)-\phi(h)\,|\, p_g\ne 0\mbox{ and } p_h\ne 0\}.\]
We extend this to all of $\Z[\G]$ by letting $\deg_\phi(0)=-\infty$. 
Since $\Z[\G]$ has no nontrivial zero divisors it follows that for $p,q\in \Z[\G]$ we have $\deg_\phi(pq)=\deg_\phi(p)+\deg_\phi(q)$.
Given $pq^{-1}\in \K(\G)$ we also define
\[ \deg_\phi(pq^{-1}):=\deg_\phi(p)-\deg_\phi(q).\]
It is straightforward to see that  this is indeed well-defined. 

%==================================================================
\subsection{Non-commutative Reidemeister torsion of presentations}

Let $X$ be a finite $CW$--complex with $G=\pi_1(X)$, and denote $\wti{X}$ the universal cover of $X$.
As in Section  \ref{section:twitorsion} we view $C_*(\wti{X})$ as a chain complex of left $\Z[G]$-modules.
Now let $\varphi\co G\to \G\in \TTEEAA$ be a homomorphism, and consider the chain complex of left $\K(G)$-modules
\[ C_*(X;\K(\G))=\K(\G) \otimes_{\Z[G]}C_*(\wti{X}),\]
where $G$ acts on $\K(\G)$ on the right via the homomorphism $\varphi$. 
If $C_*(X;\K(\G))$ is not acyclic, define the corresponding Reidemeister torsion $\tau(X,\varphi)$ to be zero.
Otherwise choose an ordering of the cells of $X$ and  for  each cell in $X$ pick a lift to $\wti{X}$. This turns $C_*(X;\K(\G))$ into a chain complex of based $\K(\G)$ left-modules and we define 
\[ \tau(X,\varphi) \in K_1(\K(\G))\]
to be the Reidemeister torsion of the based chain complex $C_*(X;\K(\G))$.
Here $K_1(\K(\G))$ is the abelianization of the direct limit $\lim\limits_{n\to \infty} \mbox{GL}(n,\K(\G))$ of the general linear groups over $\K(\G)$ (see \cite{Mi66,Ro94} for details).
We  write $\K(\G)^\times =\K(\G)\sm \{0\}$
and denote by $\K(\G)^\times_{\op{ab}}$ the abelianization of the multiplicative group $\K(\G)^\times$.
 The Dieudonn\'e determinant, see \cite{Ro94}, 
gives rise to an isomorphism $K_1(\K(\G))\to \K(\G)^\times_{\op{ab}}$ which we will use to identify these two groups.
The invariant  $\tau(X,\varphi)\in \K(\G)^\times$ is well-defined up to multiplication by an element of the form $\pm g,$ where $g\in \G$. 
Furthermore, it does not depend on the homeomorphism type of $X$.
We refer to \cite{Tu01,Fr07,FH07} for details.

It follows from $\deg_\phi(p\cdot q)=\deg(p)+\deg(q)$ for $p,q\in \K(\G)^\times$ that $\deg_\phi$ descends to a homomorphism $\deg_\phi\co \K(\G)^\times_{\op{ab}}\to \Z$. In particular $\deg_\phi(\tau(X,\varphi))$ is defined.

%==================================================================

\begin{proof}[Proof of Proposition \ref{mainprop2}]
Let  $N$ be an irreducible 3--manifold and suppose $\pi=\langle x,y\,|\, r\rangle$ is a cyclically reduced $(2,1)$--presentation of its fundamental group. Without loss of generality we may assume that $x$ represents a non-zero element in $H:=H_1(N;\Z)/\mbox{torsion}$. We need to show that   $\PP_\pi\leq \PP_N$. 

We call $\phi\in \hom(\pi,\R)$ \emph{generic} if there are vertices $v$ and $w$ of $\PP(r_y)$ such that  $\phi$ pairs maximally with $v$ and $\phi$ pairs minimally with $w$.

\begin{claim}
For any generic epimorphism $\phi\co \pi\to \Z,$ we have
\( \th_{\PP_\pi}(\phi)\leq x_N(\phi).\) 
\end{claim}

We denote by $v$ and $w$ the (necessarily unique) vertices  of $\PP(r_y)$ such that  $\phi$ pairs maximally with $v$ and  minimally with $w$.
By Corollary  \ref{cor:groupring} and Theorem~\ref{thm:ep61} there exists a homomorphism $\a\co \pi_1(N)\to \G\in\TTEEAA$  such that $\a(r_y^v\cdot r_y^w)\ne 0$. In particular, $\a(r_y^v)\ne 0$ and  $\a(r_y^w)\ne 0$.
Denote $\psi\co \pi\to H$ the canonical epimorphism. 
After possibly replacing $\a$ by $\a\times \psi$ we can and will assume that $\psi$ factors through $\a$. 
In particular $\phi$ factors through $\a$ and $\a(x)$ is a non-trivial element in $\G$.

We  denote by $X$ the $CW$--complex corresponding to the presentation $\pi$ with one 0--cell, two 1--cells corresponding to the generators $x$ and $y$ 
and one 2--cell corresponding to the relator $r$. As in \cite{FT15} we have $\tau(N,\a)=\tau(X,\a)$. 
We then have 
\[ \ba{rcl}
\th_{\PP_\pi}(\phi)&=&\th_{\PP(r_y)}(\phi)-\th_{\PP(x-1)}(\phi)\\
&=&(\phi(v)-\phi(w))-|\phi(x)|\\
&=&
\degphi(\a(r_y))-\degphi(\a(x)-1)\\
&=&\degphi(\a(r_y)\cdot \a(x-1)^{-1})\\
&=&\degphi(\tau(X,\a))=\degphi(\tau(N,\a))\leq x_N(\phi).\ea \]
Here the first two equalities follows from the definitions and the choice of $v$ and $w$. The fifth equality is \cite[Theorem~2.1]{Fr07} and the last inequality is given  by  \cite[Theorem~1.2]{Fr07} (see also \cite{Co04,Ha05,Tu02}).
This concludes the proof of the claim.\\

It is straightforward to see that the non-generic elements in $\hom(\pi,\R)$ correspond to a union of proper subspaces of $\hom(\pi,\R)$. 
By continuity  and linearity of seminorms we see  that the inequality
$ \th_{\PP_\pi}(\phi)\leq x_N(\phi)$ holds in fact
for all $\phi\in \hom(\pi,\R)$. It follows from the definitions and the discussion in Section \ref{section:convex} that
$\PP_N^{\sym}\leq \PP_N$.
\end{proof}

%==================================================================
\subsection{}
\label{section:proofmainthm}
For the reader's convenience we recall the statement of Theorem~\ref{mainthm}.

\mainthm*

%\begin{reptheorem}{mainthm}
%Let $N$ be an irreducible 3--manifold that admits  a cyclically reduced $(2,1)$--presentation  $\pi=\langle x,y\,|\, r\rangle$. Then
%\[ \mathcal{M}_N\doteq \mathcal{M}_\pi.\]
%\end{reptheorem}

\begin{proof}
It follows from Propositions \ref{mainprop1} and \ref{mainprop2} that $\PP_N\leq \PP_\pi$ and $\PP_\pi^{\sym}\leq \PP_N$. By the symmetry of the Thurston norm we also have $\PP_N=\PP_N^{\sym},$ and this implies $\PP_N\doteq \PP_\pi$.

The fact that the markings agree is an immediate consequence of \cite[Theorem~1.1]{FT15} and \cite[Theorem~E]{BNS87}.
\end{proof}

%==================================================================
\section{Examples}
\label{sec:Examples}

Currently there is no geometric characterization of those 3--manifolds whose fundamental group may be presented using only two generators and one relator. Waldhausen's question~\cite{Wald78}
of whether the rank of the fundamental group equals the Heegaard genus gives the conjectural picture that all of these manifolds have tunnel-number one. Li~\cite{Li11} gives examples of 3--manifolds whose rank is strictly smaller than the genus, including closed manifolds, manifolds with boundary, hyperbolic manifolds, and manifolds with non-trivial JSJ decomposition. See also related work of Boileau, Weidmann and Zieschang~\cite{BoZi84, BoWe05}. However, Waldhausen's question remains open for hyperbolic 3--manifolds of rank 2 and for knot complements in $S^3$.
 
\subsection{Tunnel-number one manifolds}
A \emph{tunnel-number one} 3--manifold is a 3--manifold obtained by attaching a 2--handle to a 3--dimensional 1--handlebody of genus two. The fundamental group has a presentation with two generators  from the handlebody and one relator corresponding to the attaching circle of the 2--handle.
Theorem~\ref{mainthm} allows us to compute the unit ball of the Thurston norm with ease, whilst other methods, such as normal surface theory~\cite{TW96, CT09} have limited scope (cf.~\cite{DR10}). Moreover, with Theorem~\ref{mainthm} one can easily construct examples with prescribed combinatorics or geometry of the unit ball.

Brown's algorithm is an essential ingredient in Dunfield and D.~Thurston's proof~\cite{DT05} that the probability of a tunnel-number one manifold fibring over the circle is zero. This can be paraphrased as: the probability that the unit ball has a non-empty set of marked vertices is zero. Interesting applications of Theorem~\ref{mainthm} combined with the methods of \cite{DT05} would be further predictions about the unit ball of a random tunnel-number one manifold.

\subsection{Knots or links in $\mathbf{S^3}$}
Norwood \cite{No82} showed that if the complement of a given knot in $S^3$ has fundamental group generated by two elements, then the knot is prime. The complements of \emph{tunnel-number one} knots or links in $S^3$ are tunnel-number one manifolds. This includes the 2--bridge knots and links, but Johnson~\cite{Jo06} showed that there are hyperbolic tunnel-number one knots with arbitrarily high bridge number. There is a complete classification of all tunnel-number one satellite knots by Morimoto and Sakuma \cite{MS91}, and Morimoto \cite{Mo94} also showed that a composite link has tunnel-number one if and only if it is a connected sum of a 2--bridge knot and the Hopf link.

\begin{figure}[h] 
\begin{subfigure}[t]{3.8cm}
		\centering
		\raisebox{0.9cm}
		{\includegraphics[width=3.8cm]{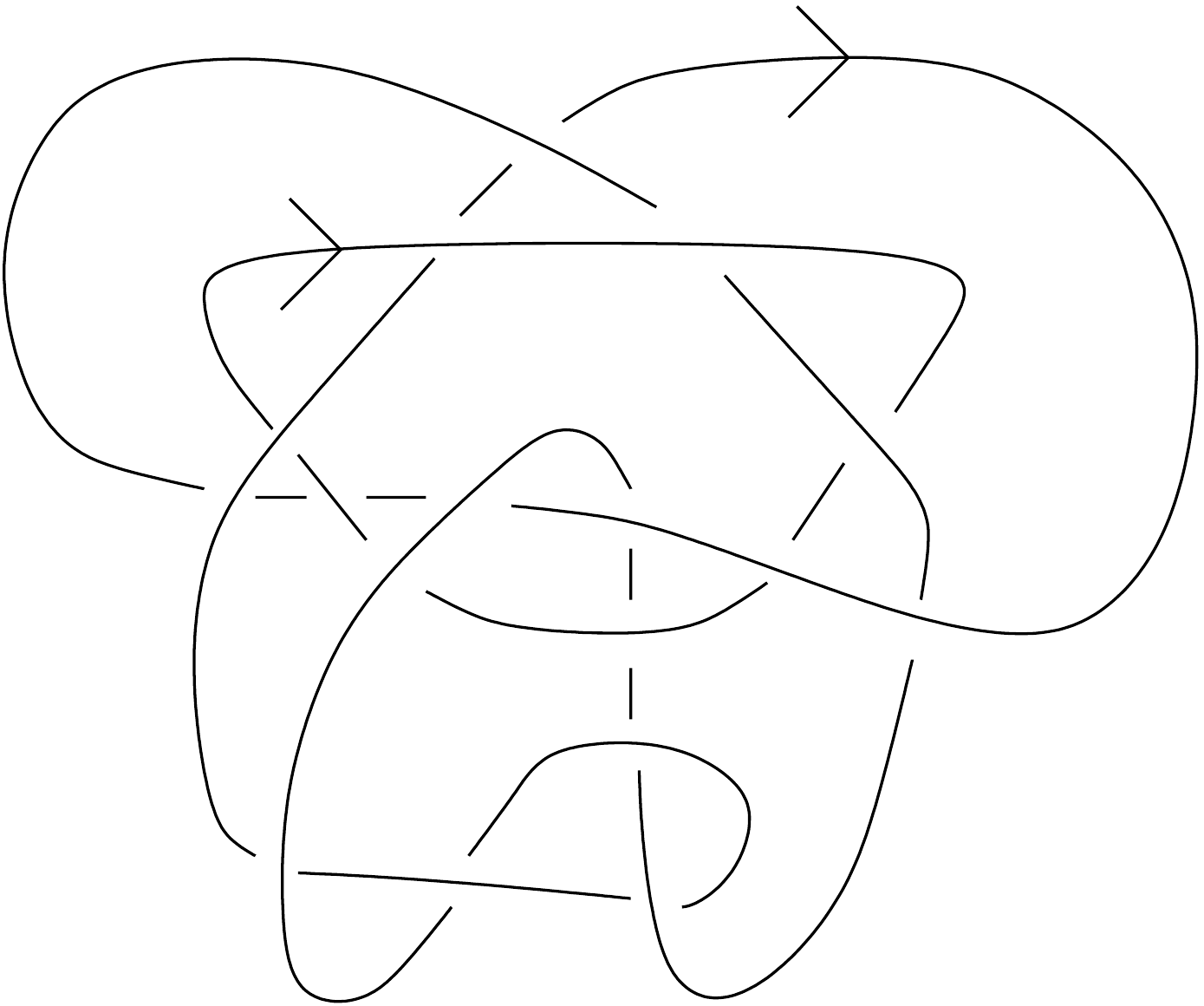}}
 		\caption{The link}\label{dunfield} \label{fig:dunfield}
		 \hfill
 \end{subfigure}
 \quad\quad
\begin{subfigure}[t]{11cm}
		\centering
		\includegraphics[width=11cm]{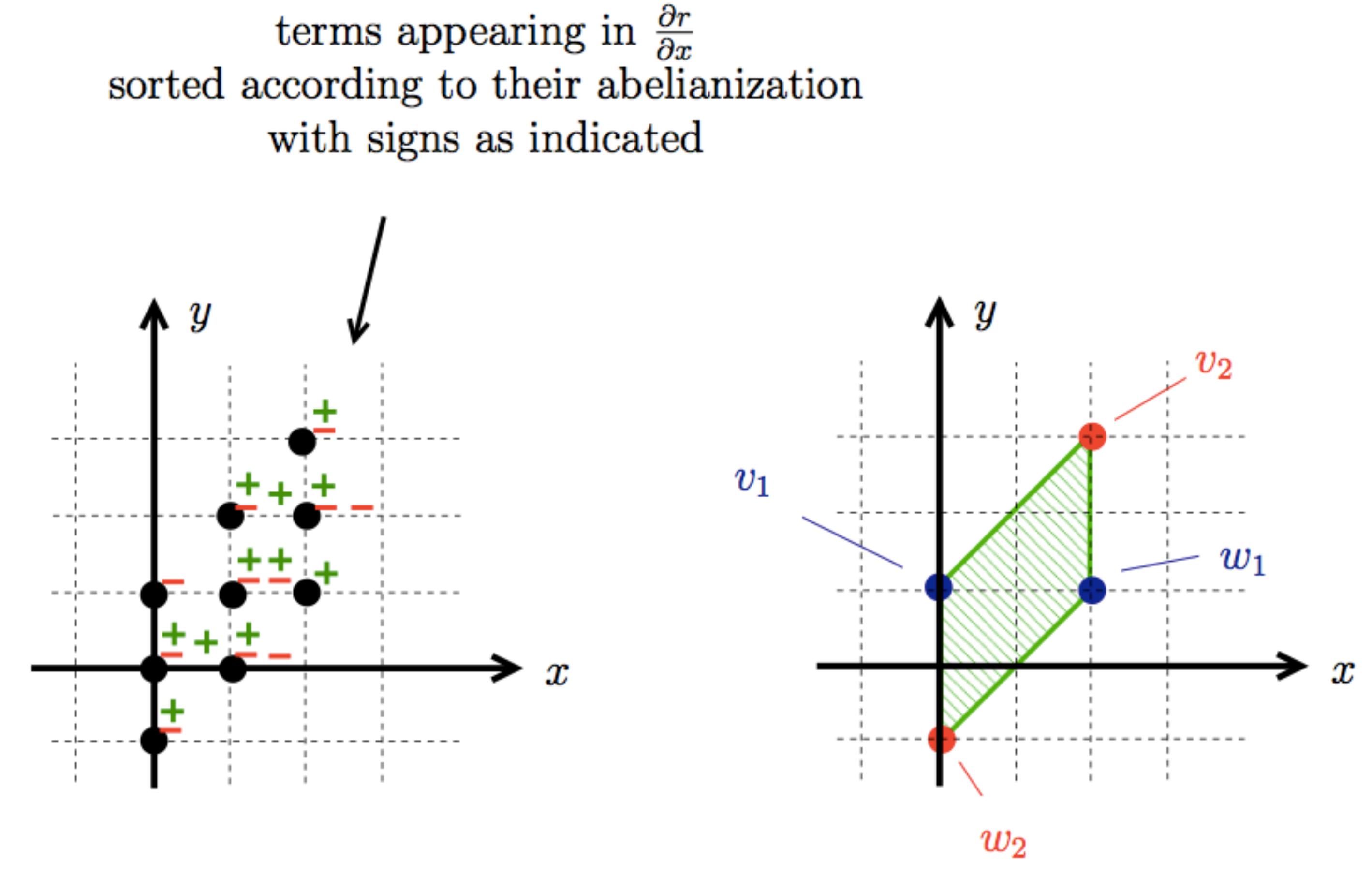}
		\caption{Calculation of $\frac{\partial r}{\partial x}$ for Dunfield's link}\label{fig:PP(r_x)}
\end{subfigure}
\caption{Dunfield's link}
\end{figure}

\subsection{Dunfield's link}\label{section:dunfield}
We conclude this section with an explicit calculation for the link $L$ shown in Figure~\ref{dunfield}, which was studied by Dunfield~\cite{Du01}.
Write $X_L:=S^3\sm \nu L$ and denote $\pi:=\pi_1(X_L)$ the link group. Then $\pi$  has the presentation
\[\ll x,y| 
x^2yx^{-1}yx^2yx^{-1}y^{-3}x^{-1}yx^2yx^{-1}yxy^{-1}x^{-2}y^{-1}xy^{-1}x^{-2}y^{-1}xy^3xy^{-1}x^{-2}y^{-1}xy^{-1}x^{-1}y\rr\]
where a meridian for the unknotted component is given by $y^{-1}x^{-1}yx^2yx^{-1}yx^2yx^{-1}y^{-3}$ and a meridian
for the other component is given by $x^{-1}y^{-1}$. Theorem \ref{mainthm} implies that $\PP_\pi\doteq \PP_N$. 

We use the map induced by $x\mapsto (1,0)$ and $y\mapsto (0,1)$ to identify $H_1(X_L;\Z)=H_1(\pi;\Z)$ with $\Z^2$. 
A straightforward calculation shows that $\PP(r_x)$ is the polytope with vertices $v_1=(0,1),v_2=(2,3),w_1=(2,1)$ and $w_2=(0,-1)$ shown in Figure~\ref{fig:PP(r_x)}. Here $v_1$ and $w_1$ are opposite vertices of $\PP(r_x)$ and $v_2$ and $w_2$ are opposite vertices of $\PP(r_x)$. Subtracting the underlying polytope of $\MM(y-1)$ from $\PP(r_x)$ gives $\PP_\pi,$ and this agrees (up to translation) with Figure~\ref{fig:dunfield-intro}. The following computation shows that the markings are the same:
{\Small
\[ \ba{lcl}
(r_x)^{v_1}&=&x^2yx^{-1}yx^2yx^{-1}y^{-3}x^{-1}yx^2yx^{-1}yxy^{-1}x^{-2}y^{-1}xy^{-1}x^{-2}y^{-1}xy^3xy^{-1}x^{-2}\\
 (r_x)^{w_1}&=&x^2yx^{-1}yx^2yx^{-1}y^{-3}x^{-1}yx\\
(r_x)^{v_2}&=&x^2yx^{-1}yx^2yx^{-1}(-1+y^{-3}x^{-1}yx^2yx^{-1}y)\\
 (r_x)^{w_2}&=&x^2yx^{-1}yx^2yx^{-1}y^{-3}x^{-1}yx^2yx^{-1}yxy^{-1}x^{-2}y^{-1}xy^{-1}x^{-2}y^{-1}(1-xy^3xy^{-1}x^{-2}y^{-1}xy^{-1}x^{-1}).\ea \]
}

%==================================================================
\section{A conjecture and a question}\label{section:questions}

\subsection{} We conjecture that Poincar\'e duality for the 3--manifold can be seen on the level of group presentations as follows:

\begin{conjecture}\label{conj:sym}
Let $\pi=\ll x,y|r\rr$ be a $(2,1)$--presentation for the fundamental group of a 3--manifold. Let $v$ be a vertex of $\PP(r_x)$. Then there exists a unique vertex $w$ of $\PP(r_y)$ such that $v$ and $w$ are opposite vertices. Furthermore we have
\[ (r_y)^{v}\equiv (-1)^{b_0(\partial N)-1}\ol{(r_x)^w}.\]
\end{conjecture}

The twisted Reidemeister torsions of \cite{FV10} can be computed in terms of Fox derivatives, and the symmetry results for twisted Reidemeister torsions proved in \cite{Ki96,HSW10,FKK12} give strong evidence towards this conjecture.

To give an explicit example, let us return to Dunfield's link.
Given the group $G$ and $p,q\in \Z[G],$ write $p\equiv q$ if there exist $g,h\in G$ such that $p=gqh$. Furthermore, denote $p\mapsto \ol{p}$ the involution of $\Z[G]$ defined by the inversion map $g\mapsto g^{-1}$ for each $g\in G$.
We denote by $\pi=\ll x,y|r\rr$ the presentation from Section \ref{section:dunfield}.
 We then note that 
\[ \ba{lcl}(r_x)^{v_2}&\equiv &-1+y^{-3}x^{-1}yx^2yx^{-1}y\\
 (r_x)^{w_2}&\equiv&1-xy^3xy^{-1}x^{-2}y^{-1}xy^{-1}x^{-1}.\ea \]
The relator $r$ is conjugate to 
\[ yx^2yx^{-1}yx^2yx^{-1}\big(y^{-3}x^{-1}yx^2yx^{-1}y\big)xy^{-1}x^{-2}y^{-1}xy^{-1}x^{-2}y^{-1}\big(xy^3xy^{-1}x^{-2}y^{-1}xy^{-1}x^{-1}\big).\]
In particular writing $s=yx^2yx^{-1}yx^2yx^{-1}$ we have the following equality in $\Z[\pi]$:
\[\ba{rcl} (r_x)^{v_2}&\equiv&s\,(r_y)^{v_2}\,s^{-1}=s(-1+y^{-3}x^{-1}yx^2yx^{-1}y)s^{-1}\\
&=&-1+\big(xy^3xy^{-1}x^{-2}y^{-1}xy^{-1}x^{-1}\big)^{-1}\\[2mm]
&=&-\ol{(r_x)^{w_2}}.\ea \]

\subsection{} 
We initially attempted to prove Theorem~\ref{mainthm} just using twisted Reidemeister torsions corresponding to finite-dimensional representations, noting that Theorem~\ref{mainthm} follows from Proposition~\ref{prop:tapdetecttn}~(1) together with an affirmative answer to the following question, which is interesting in its own right.

\begin{question}\label{qu:groupring}
Let $N$ be an aspherical 3--manifold and write $\pi=\pi_1(N)$. Let $p$ be  a non-zero element in $\Z[\pi]$. Does there exist a representation $\a\co \pi\to \gl(k,\C)$ such that $\det(\a(f))\ne 0$?
\end{question}

\end{document}

%% file: first-example-4-dunfield.pstex_t
\begin{picture}(0,0)%
\includegraphics{first-example-4-dunfield.pdf}%
\end{picture}%
%
%  Created by WinFIG version 5.04 
%  METADATA <version>1.0</version> 
%
\setlength{\unitlength}{789sp}%
\begingroup\makeatletter\ifx\SetFigFont\undefined%
\gdef\SetFigFont#1#2#3#4#5{%
  \reset@font\fontsize{#1}{#2pt}%
  \fontfamily{#3}\fontseries{#4}\fontshape{#5}%
  \selectfont}%
\fi\endgroup%
\begin{picture}(33271,14334)(2654,-19027)
%  METADATA <id>389</id> 
\put(19453,-8185){\makebox(0,0)[lb]{\smash{{\SetFigFont{12}{14.4}{\rmdefault}{\mddefault}{\updefault}{\color[rgb]{0,.69,0}unmarked vertex}%
}}}}
%  METADATA <id>390</id> 
\put(10985,-7000){\makebox(0,0)[lb]{\smash{{\SetFigFont{12}{14.4}{\rmdefault}{\mddefault}{\updefault}{\color[rgb]{0,.69,0}marked vertex}%
}}}}
%  METADATA <id>521</id> 
\put(7525,-18734){\makebox(0,0)[lb]{\smash{{\SetFigFont{12}{14.4}{\rmdefault}{\mddefault}{\updefault}{\color[rgb]{0,0,0}$(2)$ and $(3)$}%
}}}}
%  METADATA <id>522</id> 
\put(18872,-18734){\makebox(0,0)[lb]{\smash{{\SetFigFont{12}{14.4}{\rmdefault}{\mddefault}{\updefault}{\color[rgb]{0,0,0}$(4)$}%
}}}}
%  METADATA <id>523</id> 
\put(28235,-18734){\makebox(0,0)[lb]{\smash{{\SetFigFont{12}{14.4}{\rmdefault}{\mddefault}{\updefault}{\color[rgb]{0,0,0}$(5)$}%
}}}}
%  METADATA <id>35</id> 
\put(4840,-5451){\makebox(0,0)[lb]{\smash{{\SetFigFont{12}{14.4}{\rmdefault}{\mddefault}{\updefault}{\color[rgb]{0,0,1}(1) take the path determined by the relation $r$}%
}}}}
\end{picture}%

%% file: one-relator-manifolds_071815.bbl
\begin{thebibliography}{10}

\bibitem[Ag08]{Ag08}
I. Agol, {\em Criteria for virtual fibering}, J. Topol. 1 (2008), no. 2, 269--284.
\bibitem[Ag13]{Ag13}
I. Agol, {\em The virtual Haken conjecture}, with an appendix by I. Agol, D. Groves and J. Manning, 
	Documenta Math. 18 (2013), 1045--1087.
\bibitem[AFW15]{AFW15}
M. Aschenbrenner, S. Friedl and H. Wilton, {\em 3--manifold groups}, to be published by the EMS Series of Lectures in Mathematics (2015)
\bibitem[BNS87]{BNS87}
R. Bieri, W. D. Neumann and R. Strebel, {\em A geometric invariant of discrete groups}, Invent.
Math. 90 (1987), 451-477.
\bibitem[BW05]{BoWe05} M. Boileau and R. Weidmann, {\em The structure of 3--manifolds with two-generated fundamental group}, Topology 44 (2005) 283--320.
\bibitem[BZ84]{BoZi84} M. Boileau and  H. Zieschang, {\em Heegaard genus of closed orientable Seifert 3--manifolds}, Invent. Math. 76 (1984) 455--468.
\bibitem[BS05]{BS05}
S. Buyalo and P. Svetlov, {\em
Topological and geometric properties of graph-manifolds},
St. Petersburg Math. J. {\bf 16} (2005), no. 2, 297--340.
\bibitem[Co04]{Co04}
T. Cochran,  {\em Noncommutative knot theory}, Algebr. Geom. Topol. 4 (2004), 347--398.
\bibitem[COT03]{COT03}
T. Cochran, K. Orr and  P. Teichner, {\em Knot concordance, Whitney towers and $L\sp 2$-signatures},
Ann. of Math. (2) { 157} (2003),  no. 2,  433--519.
\bibitem[CT09]{CT09} D. Cooper and S. Tillmann, {\em The Thurston norm via Normal Surfaces}, 
Pac. J. Math. 239 (2009) 1--15. 
\bibitem[DL07]{DL07}
W. Dicks and P. A. Linnell, {\em $L^2$-Betti numbers of one-relator groups}, Math. Ann. 337 (2007), 855--874.
\bibitem[DLMSY03]{DLMSY03}
J. Dodziuk, P. Linnell, V. Mathai, T. Schick and S. Yates, {\em Approximating $L^2$-invariants,
and the Atiyah conjecture}, Preprint Series SFB 478 M\"unster, Germany. Communications on Pure
and Applied Mathematics  56 (2003), 839--873.
\bibitem[Du01]{Du01}
N. Dunfield, {\em Alexander and Thurston norms of fibered
3--manifolds}, Pacific J. Math. {200} (2001), no. 1,
43--58.
\bibitem[DT05]{DT05} N. Dunfield and D. Thurston, {\em Does a random tunnel-number one 3-manifold fiber over the circle?}, slides available at
\texttt{http://www.math.uiuc.edu/~nmd/preprints/slides/Treiste.pdf}
\bibitem[DR10]{DR10} N. Dunfield and D. Ramakrishnan, {\em Increasing the number of fibered faces of arithmetic
hyperbolic 3-manifolds,} Amer. J. Math. 132 (2010), no. 1, 53--97.
\bibitem[Ep61]{Ep61}
D. B. A. Epstein, {\em
Finite presentations of groups and 3--manifolds},
Quart. J. Math. Oxford Ser. (2) 12 (1961), 205--212.
\bibitem[Fo53]{Fo53}
R. H. Fox, {\em Free differential calculus I, Derivation in the free group ring}, Ann. Math. 57 (1953), 547--560.
\bibitem[Fr07]{Fr07}
S. Friedl, {\em Reidemeister torsion, the Thurston norm and Harvey's invariants},
Pac.  J. Math. 230 (2007), 271--296.
\bibitem[FH07]{FH07}
S. Friedl and S. Harvey, {\em Non--commutative multivariable Reidemeister torsion and the Thurston
norm}, Alg. Geom. Top.  7 (2007), 755--777.
\bibitem[FKm06]{FKm06}
 S. Friedl and T. Kim, \emph{Thurston norm, fibered manifolds and twisted Alexander
polynomials}, Topology  45 (2006), 929--953.
\bibitem[FKm08]{FKm08}
 S. Friedl and T. Kim, \emph{Twisted Alexander norms give lower bounds on the
Thurston norm}, Trans. Amer. Math. Soc.  360 (2008), 4597--4618.
\bibitem[FKK12]{FKK12}
 S. Friedl, T. Kim and T. Kitayama, \emph{Poincar\'e duality and degrees of twisted Alexander polynomials}, Indiana Univ. Math. J. 61 (2012), 147--192.
\bibitem[FKt14]{FKt14}
S. Friedl and T. Kitayama, {\em The virtual fibering theorem for 3--manifolds}, Enseign. Math. {60} (2014), no. 1, 79--107.
\bibitem[FT15]{FT15}
S. Friedl and S. Tillmann, {\em Two-generator one-relator groups and marked polytopes}, Preprint (2015)
 \bibitem[FV10]{FV10}
S. Friedl and  S. Vidussi,
 {\em A survey of twisted Alexander polynomials}, The Mathematics of Knots: Theory and Application (Contributions in Mathematical and Computational Sciences), editors: Markus Banagl and Denis Vogel (2010), 45--94.
\bibitem[FV12]{FV12}
S. Friedl and S. Vidussi,
{\em The Thurston norm and twisted Alexander polynomials}, Preprint (2012), to be published by J. Reine Ang. Math. 
\bibitem[Ha05]{Ha05}
S. Harvey, {\em Higher--order polynomial invariants of 3--manifolds giving lower bounds for the
Thurston norm}, Topology 44 (2005), 895--945.
\bibitem[He87]{He87}
J. Hempel, {\em Residual finiteness for 3--manifolds}, Combinatorial group theory and topology (Alta, Utah, 1984), pp. 379--396, Ann. of Math. Stud., 111, Princeton Univ. Press, Princeton, NJ, 1987.
\bibitem[Hi97]{Hi97}
J. Hillmann, {\em On $L^2$--homology and asphericity}, Israel J. Math. 99 (1997), 271--283.
\bibitem[HSW10]{HSW10}
 J. Hillman, D. Silver and S. Williams, {\em On Reciprocality of Twisted Alexander Invariants}, 
 Alg. Geom. Top.  10 (2010), 2017--2026.
\bibitem[Jo06]{Jo06} J. Johnson, {\em Bridge Number and the Curve Complex},  	arXiv:math/0603102v2.
\bibitem[Ki96]{Ki96}
T. Kitano, {\em Twisted Alexander polynomials and Reidemeister
torsion}, Pacific J. Math. 174 (1996), no. 2, 431--442.
\bibitem[KLM88]{KLM88}
P. H. Kropholler, P. A. Linnell and J. A. Moody, {\em Applications of a new $K$-theoretic theorem to
soluble group rings}, Proc. Amer. Math. Soc.  104 (1988),  no. 3, 675--684.
\bibitem[Li13]{Li11} T. Li, {\em Rank and genus of 3--manifolds}, J. Amer. Math. Soc. 26 (2013), 777--829 .
\bibitem[LS07]{LS07}
P. Linnell and P. Schick, {\em  Finite group extensions and the Atiyah conjecture}, J. Amer. Math. Soc. 20 (2007), no. 4, 1003--1051.  
\bibitem[Li13]{Li13}
Y. Liu, {\em Virtual cubulation of
nonpositively curved graph manifolds}, J. Topol. { 6} (2013), no. 4, 793--822. 
\bibitem[LS77]{LS77}
R. Lyndon and P. Schupp, {\em Combinatorial group theory}, Ergebnisse der Mathematik und ihrer Grenzgebiete, Band 89. Springer-Verlag, Berlin-New York, 1977.
\bibitem[Mi66]{Mi66}
J. Milnor, {\em  Whitehead torsion},  Bull. Amer. Math. Soc.  72 (1966), 358--426.
\bibitem[Mo94]{Mo94} K. Morimoto, {\em On composite tunnel-number one links,} Topology Appl. 59 (1994) 59--71.
\bibitem[MS91]{MS91} K. Morimoto and  M. Sakuma, {\em On unknotting tunnels for knots,} Math. Ann. 289 (1991) 143--167.
\bibitem[FN15]{FN15}
S. Friedl and M. Nagel, {\em Twisted Reidemeister torsion and the Thurston norm: graph manifolds and finite representations}, preprint (2015)
\bibitem[No82]{No82} F. H. Norwood, {\em Every two-generator knot is prime,} Proc. Amer. Math. Soc. 86 (1982) 143--147.
\bibitem[Pa77]{Pa77}
D. S. Passman, {\em The algebraic structure of group rings},
John Wiley \& Sons. XIV (1977)
\bibitem[PW14]{PW14}
 P. Przytycki and D. Wise, {\em Graph manifolds with boundary are virtually special},  J. Topology {7} (2014), 419--435.
\bibitem[PW12]{PW12}
 P. Przytycki and D. Wise, {\em Mixed 3--manifolds are virtually special}, preprint (2012).
\bibitem[RZ10]{RZ10}
 L. Ribes and P. Zalesskii,
  {\em Profinite Groups}, 2nd ed., Ergebnisse der Mathematik und ihrer Grenzgebiete, 3. Folge, vol. 40, Springer-Verlag, Berlin, 2010.
\bibitem[Ro94]{Ro94} J. Rosenberg, {\em Algebraic K -theory and its applications}, Graduate Texts in Mathematics 147, Springer, New York, 1994. 
\bibitem[Sc14]{Sc14}
K. Schreve, {\em The $L^2$-cohomology of discrete groups}, (2015), Ph.D Thesis. 
\bibitem[Se97]{Se97} J.-P. Serre, {\em Galois Cohomology}, Springer-Verlag, Berlin, 1997.
\bibitem[Th86]{Th86}
W. P. Thurston, {\em A norm for the homology of 3--manifolds}, Mem. Amer. Math. Soc. 59, no.
339 (1986), 99--130.
\bibitem[Ti70]{Ti70}
D. Tischler, {\em On fibering certain foliated manifolds over $S^1$}, Topology 9 (1970), 153--154.
\bibitem[TW96]{TW96} J. L. Tollefson and N. Wang, {\em Taut normal surfaces.} Topology 35 (1996), 55--75.
\bibitem[Tu01]{Tu01}
V. Turaev, {\em Introduction to Combinatorial Torsions}, Lectures in Mathematics, ETH Z\"urich (2001)
\bibitem[Tu02]{Tu02}
V. Turaev, {\em A homological estimate for the Thurston norm},
unpublished note (2002), arXiv:math.~GT/0207267
\bibitem[Wa94]{Wa94}
M. Wada, {\em Twisted Alexander polynomial for finitely presentable groups}, Topology 33, no. 2:
241--256 (1994)
\bibitem[Wald78]{Wald78} F. Waldhausen, {\em Some problems on 3-manifolds,} Proc. Sympos. Pure Math. 32 (1978), 313--322.
\bibitem[We72]{We72}
C. Weinbaum, {\em On relators and diagrams for groups with one defining relation},
Illinois J. Math. 16 (1972), 308--322.
\bibitem[Wi09]{Wi09}
D. Wise, {\em The structure of groups with a quasi-convex hierarchy}, Electronic Res. Ann. Math. Sci. 16 (2009), 44--55.
\bibitem[Wi12a]{Wi12a}
D. Wise, {\em The structure of groups with a quasi-convex hierarchy}, 189 pages, preprint (2012),
downloaded on October 29, 2012 from \\
\texttt{http://www.math.mcgill.ca/wise/papers.html}
\bibitem[Wi12b]{Wi12b}
D. Wise, {\em From riches to RAAGs: 3--manifolds, right--angled Artin groups, and cubical geometry}, CBMS Regional Conference Series in Mathematics, 2012.
\end{thebibliography}
